\documentclass[10pt,a4paper]{article}

\usepackage[utf8]{inputenc}

\usepackage{amsmath}
\usepackage{amsfonts}
\usepackage{amssymb}
\usepackage{mathabx}
\usepackage{bm}
\usepackage{makeidx}
\usepackage{hyperref}
\usepackage{cleveref}
\usepackage[usenames,dvipsnames]{color}
\usepackage{pagecolor,lipsum}

\usepackage{verbatim}

\usepackage{psfrag}
\usepackage{graphicx}
\usepackage{tikz}
\usetikzlibrary{patterns}

\usepackage{algorithmic}
\usepackage[ruled]{algorithm}

\usepackage{bm}
\usepackage{stmaryrd}

\newtheorem{lemma}{Lemma}[section]

\newtheorem{assumption}{Assumption}[section]

\newtheorem{gevp}{GEVP}[section]
\crefname{gevp}{GEVP}{GEVPs}

\newcommand{\R}{\mathbb R}

\newcommand{\HO}{H}
\newcommand{\HP}{H_D}
\newcommand{\RL}{{\mathcal R}}
\newcommand{\NC}{{k_0}}
\newcommand{\MC}{{k_1}}

\newcommand{\id}{{I_d}}

\newcommand{\lecnot}{\cite{Dolean:2015:IDDSiam} }

\newcommand{\Gw}{G_i^{\bot_{B_i}}}
\newcommand{\bw}{b_{G_i^{\bot_{B_i}}}}
\newcommand{\Bw}{B_{G_i^{\bot_{B_i}}}}

\newcommand{\QED}{\hspace*{\fill}\rule{2.5mm}{2.5mm}}
\newenvironment{proof}{{\bf Proof\ }}{\QED\\}

\usepackage{authblk}
\renewcommand{\id}{I}

\author[1]{N.~Bootland}
\author[2]{V.~Dolean}
\author[3]{F.~Nataf}
\author[3]{P.-H.~Tournier}

\affil[1]{\footnotesize Dept of Maths and Stats, Univ. of Strathclyde, UK}

\affil[2]{\footnotesize Dept of Maths and Stats, Univ. of Strathclyde, UK and LJAD, Univ. Côte d'Azur, France}

\affil[3]{\footnotesize Laboratoire J.L.~Lions, UPMC, CNRS UMR7598, Equipe LJLL-INRIA Alpines, Paris, France}

\title{Two-level DDM preconditioners for positive Maxwell equations}

\begin{document}

\maketitle

\setcounter{tocdepth}{5}
\tableofcontents
\pagenumbering{arabic}

\begin{abstract}
	In this paper we develop and analyse domain decomposition methods for linear systems of equations arising from conforming finite element discretisations of positive Maxwell-type equations. Convergence of domain decomposition methods rely heavily on the efficiency of the coarse space used in the second level. We design adaptive coarse spaces that complement the near-kernel space made of the gradient of scalar functions. This extends the results in~\cite{Hiptmair:2007:NSP} to the variable coefficient case and non-convex domains at the expense of a larger coarse space.
\end{abstract}

\section{Domain decomposition and preliminaries}
\label{sec:domain_decomposition_preliminaries}

\subsection{Domain decomposition}

Our notation follows that in the article on inexact coarse solves~\cite[Section 2 ``Basic definitions'']{Nataf:2020:MAR}, such as the definitions of $\NC$ and $\MC$. We will further rely on the fictitious space lemma, as given in \cite[Lemma 3.1]{Nataf:2020:MAR}, for example.

\subsection{Preliminaries}
\label{sub:preliminaries}

Let $A \in \R^{\#\mathcal{N}\times\#\mathcal{N}}$ be a symmetric positive definite matrix, stemming from the symmetric coercive bilinear form $a(\cdot,\cdot)$ on $\Omega$, and suppose it has a large near-kernel associated with the subspace $G \subset \R^{\#\mathcal{N}}$ (for instance, the gradient of $H^1$ functions on $\Omega$ when $a$ involves a curl operator). Define $G_i := R_i G$, where $R_i$ is the restriction to $\Omega_i$, and let $V_G \subset \R^{\#\mathcal{N}}$ be the vector space spanned by the sequence $(R_i^T D_i G_i)_{1\le i \le N}$, so that $G \subset V_G$. Further, let $Z \in \R^{\#\mathcal{N}_G\times \#\mathcal{N}}$ be a rectangular matrix whose columns are a basis of $V_G$ and $E$ the coarse space matrix defined by $E := Z^TAZ$.

For each subdomain $i$, let $b_i$ be a local symmetric coercive bilinear form giving rise to the matrices $B_i \in \R^{\#\mathcal{N}_i\times\#\mathcal{N}_i}$, corresponding to local subdomain solves. Within $\R^{\#\mathcal{N}_i}$ we define the $b_i$-orthogonal complement of $G_i$
\begin{equation}
\label{eq:Giortho}
\Gw := \left\lbrace \mathbf{U}_i \in \R^{\#\mathcal{N}_i} \ | \ \forall \, \mathbf{V}_i \in G_i, \ b_i(\mathbf{U}_i, \mathbf{V}_i ) = 0 \right\rbrace.
\end{equation}
It is this space that we will predominantly want to work in. As such, let $\bw$ denote the restriction of $b_i$ to $\Gw$ so that
\begin{align}
\label{eq:bw}
\bw &\colon \Gw \times \Gw \longrightarrow \R, & (\mathbf{U}_i, \mathbf{V}_i) &\mapsto b_i(\mathbf{U}_i, \mathbf{V}_i).
\end{align}
The Riesz representation theorem gives the existence of a unique isomorphism $\Bw \colon \Gw \longrightarrow \Gw$ into itself so that
\begin{align*}
\bw(\mathbf{U}_i,\mathbf{V}_i) &= (\Bw\,\mathbf{U}_i,\mathbf{V}_i) & \forall \, \mathbf{U}_i, \mathbf{V}_i &\in \Gw.
\end{align*}

An important tool that we will use is the $b_i$-orthogonal projection $\xi_{0i}$ from $\R^{\#\mathcal{N}_i}$ on $G_i$ parallel to $\Gw$. This also enables us to express the inverse of $\Bw$, which we will denote by $B_i^\dag$, through the following formula
\begin{align}
\label{eq:inverseBw}
B_i^\dag = (\id - \xi_{0i}) B_i^{-1}.
\end{align}
In order to check this formula, we have to show that
\begin{align}
\label{eq:inverseBwShow}
\Bw(\id - \xi_{0i})B_i^{-1} y= y
\end{align}
for all $y \in \Gw$. Let $z \in \Gw$, using the fact that $\id-\xi_{0i}$ is the $b_i$-orthogonal projection on $\Gw$, we have
\begin{align*}
(\Bw(\id - \xi_{0i}) B_i^{-1} y, z) = b_i((\id - \xi_{0i}) B_i^{-1} y, z) = b_i(B_i^{-1} y, z) = (y,z).
\end{align*}
Since this equality holds for any $z \in \Gw$, this proves that \eqref{eq:inverseBwShow} holds and thus that $B_i^\dag$ provides the inverse of $\Bw$.

\section{An additive Schwarz method}
\label{sec:additive_schwarz_method}

Within this section we consider the additive Schwarz method and how it can be suitably modified for our purposes. Our main premise is that the matrix $A$, defining the linear system we wish to solve, has a large near-kernel which can cause problems numerically. Thus we would like to deal with it directly in the coarse space in order that we can primarily work in the orthogonal complement, avoiding the ill-conditioning in $A$ from the near-kernel.

The additive Schwarz method is characterised by the choice of the subdomain matrices $B_i$ being given by the Dirichlet matrices $R_i A R_i^T$, along with a suitable choice of the operator $\RL$ in the fictitious space lemma.

\subsection{The underlying additive Schwarz method}
\label{sub:as}

We now define the abstract framework for the additive Schwarz preconditioner, where the choice $B_i = R_i A R_i^T$ is made. Let $\HP$ be defined as
\begin{align*}
\HP := \R^{\#\mathcal{N}_G} \times \prod_{i=1}^N \R^{\#\mathcal{N}_i}
\end{align*}
endowed with the Euclidean scalar product. For $\mathcal{U} = (\mathbf{U}_0, (\mathbf{U}_i)_{1 \le i \le N}) \in \HP$ and $\mathcal{V} = (\mathbf{V}_0, (\mathbf{V}_i)_{1 \le i \le N}) \in \HP$, with $\mathbf{U}_0,\mathbf{V}_0 \in \R^{\#\mathcal{N}_G}$ and $\mathbf{U}_i,\mathbf{V}_i \in\R^{\#\mathcal{N}_i}$ for $1 \le i \le N$, define the bilinear form $b \colon \HP \times \HP \longrightarrow \R$ arising from the coarse operator $E$ and the local SPD matrices $(B_i)_{1\le i \le N}$, where $B_i = R_iAR_i^T$, such that
\begin{align*}
(\mathcal{U},\mathcal{V}) \mapsto b(\mathcal{U},\mathcal{V}) := (E\mathbf{U}_0,\mathbf{V}_0) + \sum_{i=1}^{N} (R_i A R_i^T \mathbf{U}_i, \mathbf{V}_i).
\end{align*}
Further, we denote by $B \colon \HP \longrightarrow \HP$ the block diagonal operator such that $(B\mathcal{U},\mathcal{V}) = b(\mathcal{U},\mathcal{V})$ for all $\mathcal{U},\mathcal{V} \in \HP$. Now, using the $A$-orthogonal projection $P_0$ on $V_G$, for any $\mathcal{U} = (\mathbf{U}_0, (\mathbf{U}_i)_{1 \le i \le N}) \in \HP$ we define the linear operator $\RL_{AS} \colon \HP \longrightarrow \HO$ by
\begin{align*}
\RL_{AS}(\mathcal{U}) : = Z\mathbf{U}_0 + (\id-P_0) \sum_{i=1}^{N} R_i^T \mathbf{U}_i.
\end{align*}

In order to apply the fictitious space lemma three assumptions have to be checked.\\

$\bullet$ $\RL_{AS}$ is onto:\\
Let $\mathbf{U} \in \HO$, we have
\begin{align*}
\mathbf{U} &= P_0 \mathbf{U} + (\id - P_0) \mathbf{U} \\
&= P_0 \mathbf{U} + (\id - P_0) \textstyle\sum_{i=1}^N R_i^T D_i R_i \mathbf{U}.
\end{align*}
Now since $P_0 \mathbf{U} \in V_G$ there exists $\mathbf{U}_0$ such that $Z \mathbf{U}_0 = P_0 \mathbf{U}$. Therefore, we have
\begin{align*}
\mathbf{U} = \RL_{AS}(\mathbf{U}_0, (D_i R_i \mathbf{U})_{1\leq i\leq N}).
\end{align*}

$\bullet$ Continuity of $\RL_{AS}$:\\
We have to estimate a constant $c_R$ such that for all $\mathcal{U} = (\mathbf{U}_0, (\mathbf{U}_i)_{1 \le i \le N}) \in \HP$ we have
\begin{align*}
a( \RL_{AS} (\mathcal{U}), \RL_{AS} (\mathcal{U})) &\le c_R \, b(\mathcal{U}, \mathcal{U}) = c_R \, [ (E\mathbf{U}_0, \mathbf{U}_0) + \sum_{i=1}^N (R_i A R_i^T \mathbf{U}_i, \mathbf{U}_i)].
\end{align*}
Now we have the following estimate using the $A$-orthogonality of $\id-P_0$ and Lemma~7.9 in \lecnot (page~171):
\begin{align*}
a(\RL_{AS}(\mathcal{U}), \RL_{AS}(\mathcal{U})) &= \| Z \mathbf{U}_0 + (\id-P_0) \textstyle\sum_{i=1}^N R_i^T \mathbf{U}_i \|_A^2\\
&= \| Z \mathbf{U}_0 \|_A^2 + \| (\id-P_0) \textstyle\sum_{i=1}^N R_i^T \mathbf{U}_i \|_A^2\\
&\le (E \mathbf{U}_0, \mathbf{U}_0) + \NC \textstyle\sum_{i=1}^N \| R_i^T \mathbf{U}_i \|_A^2\\
&\le \NC \left[ (E \mathbf{U}_0, \mathbf{U}_0) + \textstyle\sum_{i=1}^N \| R_i^T \mathbf{U}_i \|_A^2 \right].
\end{align*}
Thus the estimate of the constant of continuity of $\RL_{AS}$ can be chosen as
\begin{align*}
c_R := \NC.
\end{align*}

$\bullet$ Stable decomposition with $\RL_{AS}$:\\
Let $\mathbf{U} \in \HO$, we have
\begin{align*}
\mathbf{U} &= P_0 \mathbf{U} + (\id - P_0) \mathbf{U} \\
&= P_0 \mathbf{U} + (\id - P_0) \textstyle\sum_{i=1}^N R_i^T D_i R_i \mathbf{U} \\
&= P_0 \mathbf{U} + (\id - P_0) \textstyle\sum_{i=1}^N R_i^T D_i (\id - \xi_{0i}) R_i \mathbf{U} + \underbrace{(\id - P_0) \textstyle\sum_{i=1}^N R_i^T D_i \xi_{0i} R_i \mathbf{U}}_{=0}.
\end{align*}
Let us consider the last equality: since $P_0 \mathbf{U} \in V_G$ there exists $\mathbf{U}_0$ such that $Z \mathbf{U}_0 = P_0 \mathbf{U}$, meanwhile the third term is zero since $\sum_{i=1}^N R_i^T D_i \xi_{0i} R_i \mathbf{U} \in V_G$. Therefore, we have
\begin{align*}
\mathbf{U} = \RL_{AS}(\mathbf{U}_0, (D_i (\id - \xi_{0i}) R_i \mathbf{U})_{1\leq i\leq N}).
\end{align*}
Determining the stability of this decomposition consists in estimating a constant $c_T > 0$ such that
\begin{align*}
c_T \, [(E\mathbf{U}_0, \mathbf{U}_0) + \sum_{j=1}^N (R_j A R_j^T D_j (\id-\xi_{0j}) R_j \mathbf{U}, D_j (\id-\xi_{0j}) R_j \mathbf{U})] \leq a(\mathbf{U}, \mathbf{U}).
\end{align*}
We have
\begin{align}
\nonumber
\textstyle\sum_{j=1}^N (R_j A R_j^T D_j (\id-\xi_{0j}) R_j \mathbf{U}, D_j (\id-\xi_{0j}) R_j \mathbf{U}) &\le \tau_0 \textstyle\sum_{j=1}^N (A_j^{Neu} R_j \mathbf{U}, R_j \mathbf{U})\\
\label{eq:estimateAS}
&\le \tau_0 \MC \, a(\mathbf{U}, \mathbf{U}),
\end{align}
where
\begin{align}
\label{eq:tauOneLevelAS}
\tau_0 := \max_{1 \le j \le N} \max_{\mathbf{V}_j \in \R^{\#\mathcal{N}_j}} \frac{(R_j A R_j^T D_j (\id-\xi_{0j}) \mathbf{V}_j,  D_j (\id-\xi_{0j}) \mathbf{V}_j)}{(A_j^{Neu} \mathbf{V}_j,  \mathbf{V}_j)},
\end{align}
and in the second step we have used Lemma~7.13 in \lecnot (page~175). By applying \eqref{eq:estimateAS}, we obtain
\begin{align*}
& (E\mathbf{U}_0, \mathbf{U}_0) + \textstyle\sum_{j=1}^N (R_j A R_j^T D_j (\id-\xi_{0j}) R_j \mathbf{U},  D_j (\id-\xi_{0j}) R_j \mathbf{U})\\
&\leq a(P_0\mathbf{U}, P_0\mathbf{U}) + \tau_0 \MC \, a(\mathbf{U}, \mathbf{U})\\
&\le (1 + \MC \tau_0) \, a(\mathbf{U}, \mathbf{U}).
\end{align*}
Thus we can take
\begin{align*}
c_T := \frac{1}{1 + \MC \tau_0}.
\end{align*}
Finally, the condition number estimate given by the fictitious space lemma for the induced preconditioner is
\begin{align*}
\kappa(M_{AS}^{-1} A) \le (1 + \MC \tau_0) \NC.
\end{align*}
The development that now follows is motivated by the fact that the number $\tau_0$ may be very large due to the shape and size of the domain or the heterogeneities in the coefficients of the problem. This leads to a bad condition number estimate. The fix is to enlarge the coarse space by the generalized eigenvalue problem (GEVP) induced by its formula~\eqref{eq:tauOneLevelAS}. More precisely, we introduce in the next subsection a generalized eigenvalue problem and the related GenEO coarse space.

\subsection{Additive Schwarz with GenEO}
\label{sub:as_geneo}

\begin{gevp}[Generalized Eigenvalue Problem for the lower bound]
\label{gevp:tauthresholdAS}
For each subdomain $1 \le j \le N$, we introduce the generalized eigenvalue problem
\begin{align}
\label{eq:eigAtildeBAS}
\begin{array}{c}
\text{Find } (\mathbf{V}_{jk},\lambda_{jk}) \in \R^{\#\mathcal{N}_j} \setminus \{0\} \times \R
\mbox{ such that}\\[1ex]
(\id-\xi_{0j}^T) D_j R_j A R_j^T D_j (\id-\xi_{0j}) \mathbf{V}_{jk} = \lambda_{jk} A_j^{Neu} \mathbf{V}_{jk}.
\end{array}
\end{align}
Let $\tau>0$ be a user-defined threshold, we define $V_{j,geneo}^\tau \subset \R^{\#\mathcal{N}}$ as the vector space spanned by the family of vectors $(R_j^T D_j (\id - \xi_{0j}) \mathbf{V}_{jk})_{\lambda_{jk} > \tau}$ corresponding to eigenvalues larger than $\tau$. Let $V_{geneo}^\tau$ be the vector space spanned by the collection over all subdomains of vector spaces $(V_{j,geneo}^\tau)_{1 \le j \le N}$.
\end{gevp}

In the theory that follows for the stable decomposition estimate but not in the algorithm itself, we will make use of $\pi_j$ defined as the projection from $R^{\#\mathcal{N}_j}$ on $V_{j,\tau} := \text{Span}\{ \mathbf{V}_{jk} |\, \lambda_{jk} > \tau \}$ parallel to $\text{Span}\{ \mathbf{V}_{jk} |\, \lambda_{jk} \le \tau \}$. The key to \Cref{gevp:tauthresholdAS} is the following bound, derived from Lemma~7.7 in \cite{Dolean:2015:IDDSiam} (page~168):  for all $\mathbf{U}_j \in R^{\#\mathcal{N}_j}$
\begin{align}
\label{eq:gevpLowerBoundAS}
(R_j A R_j^T D_j (\id-\xi_{0j}) (\id-\pi_j) \mathbf{U}_j, D_j (\id-\xi_{0j}) (\id-\pi_j) \mathbf{U}_j) &\le \tau\, (A_j^{Neu} \mathbf{U}_j, \mathbf{U}_j).
\end{align}

We can now build the coarse space $V_0$ from the near-kernel $G$ along with \Cref{gevp:tauthresholdAS}, defining the following vector space sum:
\begin{align}
\label{eq:V0GenEOAS}
V_0 := V_G + V_{geneo}^\tau.
\end{align}
The coarse space $V_0$ is spanned by the columns of a full rank rectangular matrix with $\#\mathcal{N}_0$ columns, which we will denote by $Z$.

We can now define the abstract framework for the additive Schwarz with GenEO preconditioner. Let $\HP$ be defined by
\begin{align*}
\HP := \R^{\#\mathcal{N}_0} \times \prod_{i=1}^N \R^{\#\mathcal{N}_i}
\end{align*}
endowed with the Euclidean scalar product. We now define the bilinear form $b \colon \HP \times \HP \longrightarrow \R$ arising from the coarse operator $E = Z^T A Z$ and the local SPD matrices $(B_i)_{1\le i \le N}$, where $B_i = R_iAR_i^T$, such that
\begin{align*}
(\mathcal{U},\mathcal{V}) \mapsto b(\mathcal{U},\mathcal{V}) := (E\mathbf{U}_0,\mathbf{V}_0) + \sum_{i=1}^{N} (R_i A R_i^T \mathbf{U}_i, \mathbf{V}_i).
\end{align*}
Further, we denote by $B \colon \HP \longrightarrow \HP$ the block diagonal operator such that $(B\mathcal{U},\mathcal{V}) = b(\mathcal{U},\mathcal{V})$ for all $\mathcal{U},\mathcal{V} \in \HP$. Using the $A$-orthogonal projection $P_0$ on $V_0$, we define $\RL_{AS,2} \colon \HP \longrightarrow \HO$ as
\begin{align*}
\RL_{AS,2}(\mathcal{U}) := Z \mathbf{U}_0 + (\id - P_0) \sum_{i=1}^N R_i^T \mathbf{U}_i.
\end{align*}

In order to apply the fictitious space lemma we check the three assumptions as before. Surjectivity and continuity of $\RL_{AS,2}$ follows identically to that for $\RL_{AS}$ as in \Cref{sub:as}. Thus we need only consider the stable decomposition.

$\bullet$ Stable decomposition with $\RL_{AS,2}$:\\
Let $\mathbf{U} \in \HO$, we have
\begin{align*}
\mathbf{U} &= P_0 \mathbf{U} + (\id - P_0) \textstyle\sum_{i=1}^N R_i^T D_i R_i \mathbf{U} \\
&= P_0 \mathbf{U} + (\id - P_0) \textstyle\sum_{i=1}^N R_i^T D_i (\id - \xi_{0i}) R_i \mathbf{U} + \underbrace{(\id - P_0) \textstyle\sum_{i=1}^N R_i^T D_i \xi_{0i} R_i \mathbf{U}}_{=0}\\
&= P_0 \mathbf{U} + (\id - P_0) \textstyle\sum_{i=1}^N R_i^T D_i (\id - \xi_{0i}) (\id - \pi_i) R_i \mathbf{U}\\
& \mathrel{\phantom{=}} \mathrel+ \underbrace{(\id - P_0) \textstyle\sum_{i=1}^N R_i^T D_i (\id - \xi_{0i}) \pi_i R_i \mathbf{U}}_{=0}.
\end{align*}
The very last term is zero since for all $i$, $R_i^T D_i (\id - \xi_{0i}) \pi_i R_i \mathbf{U} \in V_{geneo}^\tau \subset V_0$. Let $\mathbf{U}_0 \in \R^{\#\mathcal{N}_0}$ be such that $Z \mathbf{U}_0 = P_0\mathbf{U}$, then we can choose the decomposition
\begin{align*}
\mathbf{U} = \RL_{AS,2}(\mathbf{U}_0, (D_i (\id - \xi_{0i}) (\id - \pi_i) R_i \mathbf{U})_{1 \leq i \leq N}).
\end{align*}
With this decomposition, using the GEVP bound \eqref{eq:gevpLowerBoundAS} and Lemma~7.13 in \lecnot (page~175), we have
\begin{align*}
& (E\mathbf{U}_0, \mathbf{U}_0) + \textstyle\sum_{j=1}^N (R_j A R_j^T D_j (\id-\xi_{0j}) (\id - \pi_j) R_j \mathbf{U},  D_j (\id-\xi_{0j}) (\id - \pi_j) R_j \mathbf{U})\\
&\le a(Z\mathbf{U}_0, Z\mathbf{U}_0) + \tau \textstyle\sum_{j=1}^N (A_j^{Neu} R_j \mathbf{U}, R_j \mathbf{U})\\
&\le a(P_0\mathbf{U}, P_0\mathbf{U}) + \tau \MC \, a(\mathbf{U}, \mathbf{U})\\
&\le (1 + \MC \tau) \, a(\mathbf{U}, \mathbf{U}).
\end{align*}
Thus the stable decomposition property holds with constant
\begin{align*}
c_T := \frac{1}{1 + \MC \tau}.
\end{align*}

Altogether, for a given user-defined positive constant $\tau$, the additive Schwarz with GenEO preconditioner yields the condition number estimate
\begin{align*}
\kappa(M^{-1}_{AS,2} A) \le (1 + \MC \tau) \NC,
\end{align*}
where
\begin{align*}
M^{-1}_{AS,2} &= \RL_{AS,2} B^{-1} \RL_{AS,2}^* \\
&= Z (Z^T A Z)^{-1} Z^T + (\id-P_0) \sum_{i=1}^N R_i^T B_i^{-1} R_i (\id-P_0^T) \\
&= Z (Z^T A Z)^{-1} Z^T + (\id-P_0) \sum_{i=1}^N R_i^T (R_i A R_i^T)^{-1} R_i (\id-P_0^T).
\end{align*}

\subsection{Inexact coarse solves}
\label{sub:as_geneo_with_inexact_coarse_solves}

In practice, the coarse operator $E = Z^T A Z$ will be large, as it involves both the near-kernel and the GenEO coarse space, and thus solving the linear systems involving $E$ presents a bottleneck within the two-level algorithm. As such, we now consider the use of an inexact coarse solve, that is when we replace $E$ by a cheaper approximation
\begin{align*}
\tilde{E} \approx E,
\end{align*}
in order to ameliorate this factor. We are then interested in the robustness of the approach when inexact coarse solves are employed.

We make the following assumption on $\tilde{E}$ throughout:
\begin{assumption}
\label{ass:spd_tildeE}
The operator $\tilde{E}$ is symmetric positive definite.
\end{assumption}
Let us also introduce the operator $\tilde{P}_0$
\begin{align*}
\tilde{P}_0 := Z \tilde{E}^{-1} Z^T A,
\end{align*}
which approximates the $A$-orthogonal projection $P_0$ on $V_0$
\begin{align*}
P_0 = Z E^{-1} Z^T A.
\end{align*}
Note that $\tilde{P}_0$ is not a projection but, from Lemma~1 of \cite{Nataf:2020:MAR}, has the same range and kernel as $P_0$.

We now consider the additive Schwarz with GenEO preconditioner with inexact coarse solves. We utilise the same framework as in \Cref{sub:as_geneo} but now make use of the bilinear form $\tilde{b} \colon \HP \times \HP \longrightarrow \R$, with block diagonal matrix form $\tilde{B}$, such that
\begin{align*}
\tilde{b}(\mathcal{U},\mathcal{V}) := (\tilde{E}\mathbf{U}_0, \mathbf{V}_0) + \sum_{i=1}^{N} (R_i A R_i^T \mathbf{U}_i, \mathbf{V}_i),
\end{align*}
and consider the linear operator $\widetilde{\RL}_{AS,2} \colon \HP \longrightarrow \HO$ defined by
\begin{align*}
\widetilde{\RL}_{AS,2}(\mathcal{U}) : = Z\mathbf{U}_0 + (\id-\tilde{P}_0) \sum_{i=1}^{N} R_i^T \mathbf{U}_i.
\end{align*}
Note that
\begin{align*}
\widetilde{\RL}_{AS,2}(\mathcal{U}) - \RL_{AS,2}(\mathcal{U}) = (P_0-\tilde{P}_0) \sum_{i=1}^{N} R_i^T \mathbf{U}_i \in V_0
\end{align*}
since $\mathrm{Im}(P_0-\tilde{P}_0) \subset V_0$.

Before continuing, we give a lemma which will prove useful in what follows.

\begin{lemma}
\label{lemma:EandtildeE}
Let $\mathbf{U}_0 \in \R^{\#{\mathcal N}_0}$, then
\begin{align}
\label{eq:tildeEBoundedByE}
(\tilde{E} {\mathbf U}_0,{\mathbf U}_0) &\le \lambda_{max}(E^{-1} \tilde{E}) \| Z\mathbf{U}_0 \|_{A}^{2}
\end{align}
and
\begin{align}
\label{eq:EBoundedBytildeE}
\| Z\mathbf{U}_0 \|_{A}^{2} &\le \lambda_{max}(E\tilde{E}^{-1}) ( \tilde{E}\mathbf{U}_0, \mathbf{U}_0 ).
\end{align}
\end{lemma}
\begin{proof}
First, for \eqref{eq:tildeEBoundedByE} we have that
\begin{align*}
(\tilde{E} {\mathbf U}_0,{\mathbf U}_0) &= (E^{-1/2} \tilde{E} E^{-1/2} E^{1/2} {\mathbf U}_0, E^{1/2} {\mathbf U}_0)\\
&\le \lambda_{max}(E^{-1/2} \tilde{E} E^{-1/2}) (E^{1/2} {\mathbf U}_0, E^{1/2} {\mathbf U}_0)\\
&= \lambda_{max}(E^{-1} \tilde{E}) (Z^T A Z {\mathbf U}_0, {\mathbf U}_0)\\
&= \lambda_{max}(E^{-1} \tilde{E}) \|Z\mathbf{U}_0\|_{A}^2.
\end{align*}
In a similar manner
\begin{align*}
(\tilde{E} {\mathbf U}_0,{\mathbf U}_0) &= (E^{-1/2} \tilde{E} E^{-1/2} E^{1/2} {\mathbf U}_0, E^{1/2} {\mathbf U}_0)\\
&\ge \lambda_{min}(E^{-1/2} \tilde{E} E^{-1/2}) (E^{1/2} {\mathbf U}_0, E^{1/2} {\mathbf U}_0)\\
&= \lambda_{min}(E^{-1} \tilde{E}) \|Z\mathbf{U}_0\|_{A}^2
\end{align*}
and thus we deduce \eqref{eq:EBoundedBytildeE} from the equivalence
\begin{align*}
\lambda_{min}(E^{-1} \tilde{E})^{-1} = \lambda_{max}(\tilde{E}^{-1} E) = \lambda_{max}(E \tilde{E}^{-1}),
\end{align*}
the final equality arising since $\tilde{E}^{-1} E$ and $E \tilde{E}^{-1}$ are similar matrices.
\end{proof}

Predominantly following the arguments presented in \cite{Nataf:2020:MAR}, we now check the three assumptions of the fictitious space lemma.

$\bullet$ $\widetilde{\RL}_{AS,2}$ is onto:\\
For $\mathbf{U} \in \HO$ we have that
\begin{align*}
\mathbf{U} &= \tilde{P}_0 \mathbf{U} + (\id - \tilde{P}_0) \mathbf{U}\\
&= \tilde{P}_0 \mathbf{U} + (\id-\tilde{P}_0) \textstyle\sum_{i=1}^{N} R_i^T D_i R_i \mathbf{U}.
\end{align*}
Let $\mathbf{U}_0 \in \R^{\#{\mathcal N}_0}$ be such that $Z \mathbf{U}_0 = \tilde{P}_0\mathbf{U}$, then we have the decomposition
\begin{align*}
\mathbf{U} = \widetilde{\RL}_{AS,2}(\mathbf{U}_0,(D_i R_i \mathbf{U})_{1 \le i \le N}).
\end{align*}

$\bullet$ Continuity of $\widetilde{\RL}_{AS,2}$:\\
Let $\delta>0$ be a positive parameter. For $\mathcal{U} = (\mathbf{U}_0, (\mathbf{U}_i)_{1 \le i \le N}) \in \HP$, by using the fact that $\mathrm{Im}(P_0-\tilde{P}_0)$ is $a$-orthogonal to $\mathrm{Im}(\id-P_0)$, the Cauchy--Schwarz inequality, Young's inequality (with parameter $\delta$), $A$-orthogonality of $I-P_0$, the bound \eqref{eq:EBoundedBytildeE} from \Cref{lemma:EandtildeE}, and Lemma~7.9 of \lecnot (page~171) we have
\begin{align*}
& a(\widetilde{\RL}_{AS,2}(\mathcal{U}),\widetilde{\RL}_{AS,2}(\mathcal{U}))\\
&= \| \RL_{AS,2}(\mathcal{U}) + (P_0-\tilde{P}_0) \textstyle\sum_{i=1}^{N} R_i^T \mathbf{U}_i \|_{A}^{2}\\
&= \| \RL_{AS,2}(\mathcal{U}) \|_{A}^{2} + 2 \, a(\RL_{AS,2}(\mathcal{U}),(P_0-\tilde{P}_0) \textstyle\sum_{i=1}^{N} R_i^T \mathbf{U}_i))\\
& \mathrel{\phantom{=}} \mathrel+ \| (P_0-\tilde{P}_0) \textstyle\sum_{i=1}^{N} R_i^T \mathbf{U}_i \|_{A}^{2}\\
&= \| Z {\mathbf U}_0 \|_{A}^{2} + \| (\id-P_0) \textstyle\sum_{i=1}^N R_i^T {\mathbf U}_i \|_{A}^{2} + 2 \, a(Z\mathbf{U}_0,(P_0-\tilde{P}_0) \textstyle\sum_{i=1}^{N} R_i^T \mathbf{U}_i)\\
& \mathrel{\phantom{=}} \mathrel+ \| (P_0-\tilde{P}_0) \textstyle\sum_{i=1}^{N} R_i^T \mathbf{U}_i \|_{A}^{2}\\
&\le \| Z\mathbf{U}_0 \|_{A}^{2} + \| \textstyle\sum_{i=1}^N R_i^T {\mathbf U}_i \|_{A}^{2} + \delta \| Z\mathbf{U}_0 \|_{A}^{2} + \tfrac{1}{\delta} \| (P_0-\tilde{P}_0) \textstyle\sum_{i=1}^{N} R_i^T \mathbf{U}_i \|_{A}^{2})\\
& \mathrel{\phantom{=}} \mathrel+ \| (P_0-\tilde{P}_0) \textstyle\sum_{i=1}^{N} R_i^T \mathbf{U}_i \|_{A}^{2}\\
&\le (1+\delta) \| Z\mathbf{U}_0 \|_{A}^{2} + \left(1 + (1+\tfrac{1}{\delta}) \| P_0-\tilde{P}_0 \|_{A}^{2} \right) \| \textstyle\sum_{i=1}^{N} R_i^T \mathbf{U}_i \|_{A}^{2}\\
&\le (1+\delta) \lambda_{max}(E\tilde{E}^{-1}) ( \tilde{E}\mathbf{U}_0, \mathbf{U}_0 ) + \left(1 + (1+\tfrac{1}{\delta}) \| P_0-\tilde{P}_0 \|_{A}^{2} \right) \NC \textstyle\sum_{i=1}^{N} \| R_i^T \mathbf{U}_i \|_{A}^{2}\\
&\le c_R \, \tilde{b}(\mathcal{U},\mathcal{U}),
\end{align*}
for
\begin{align*}
c_R = \max\left((1+\delta) \lambda_{max}(E\tilde{E}^{-1}), \left(1 + (1+\tfrac{1}{\delta}) \| P_0-\tilde{P}_0 \|_{A}^{2} \right) \NC \right).
\end{align*}
As in \cite{Nataf:2020:MAR}, we can minimise over $\delta>0$ by equating the two terms using
\begin{align*}
\min_{\delta>0} \max(c+\alpha\delta,d+\beta\delta^{-1}) = \frac{d+c + \sqrt{(d-c)^2 + 4\alpha\beta}}{2},
\end{align*}
when all parameters are positive. Let us define $\epsilon_{A} := \| P_0-\tilde{P}_0 \|_{A}$ along with $\lambda_{+} := \lambda_{max}(E\tilde{E}^{-1})$, then we can take
\begin{align}
\label{eq:c_R_inexact_as}
c_R = \frac{\NC (1+\epsilon_{A}^2) + \lambda_{+} + \sqrt{(\NC (1+\epsilon_{A}^2)-\lambda_{+})^2 + 4 \lambda_{+} \NC \epsilon_{A}^2}}{2}.
\end{align}
Further note that Lemma~4 from \cite{Nataf:2020:MAR} gives that
\begin{align*}
\epsilon_{A} = \max\left(|1-\lambda_{min}(E\tilde{E}^{-1})|,|1-\lambda_{max}(E\tilde{E}^{-1})|\right),
\end{align*}
which allows for \eqref{eq:c_R_inexact_as} to be given solely in terms of the constant $\NC$ and the minimal and maximal eigenvalues of $E\tilde{E}^{-1}$.

$\bullet$ Stable decomposition with $\widetilde{\RL}_{AS,2}$:\\
For $\mathbf{U} \in \HO$ we have that
\begin{align*}
\mathbf{U} &= P_0 \mathbf{U} + (\id - P_0) \mathbf{U} = P_0 \mathbf{U} + (\id-P_0) \textstyle\sum_{i=1}^{N} R_i^T D_i R_i \mathbf{U}\\
&= P_0 \mathbf{U} + (\id-P_0) \textstyle\sum_{i=1}^{N} R_i^T D_i (\id-\xi_{0i}) R_i \mathbf{U}\\
&= P_0 \mathbf{U} + (\id-P_0) \textstyle\sum_{i=1}^{N} R_i^T D_i (\id-\xi_{0i}) (\id-\pi_i) R_i \mathbf{U}\\
&= F \mathbf{U} + (\id-\tilde{P}_0) \textstyle\sum_{i=1}^{N} R_i^T D_i (\id-\xi_{0i}) (\id-\pi_i) R_i \mathbf{U},
\end{align*}
where
\begin{align}
\label{eq:FU}
F \mathbf{U} = P_0 \mathbf{U} + (\tilde{P}_0-P_0) \textstyle\sum_{i=1}^{N} R_i^T D_i (\id-\xi_{0i}) (\id-\pi_i) R_i \mathbf{U} \in V_0.
\end{align}
Let $\mathbf{U}_0 \in \R^{\#{\mathcal N}_0}$ be such that $Z \mathbf{U}_0 = F\mathbf{U}$, then we have the decomposition
\begin{align*}
\mathbf{U} = \widetilde{\RL}_{AS,2}(\mathbf{U}_0,(D_i (\id-\xi_{0i}) (\id-\pi_i) R_i \mathbf{U})_{1 \le i \le N}) =: \widetilde{\RL}_{AS,2}(\mathcal{U}).
\end{align*}
We now show that this decomposition is stable, again following an analogous approach to \cite{Nataf:2020:MAR}. Firstly, note that the GEVP bound \eqref{eq:gevpLowerBoundAS} and Lemma~7.13 in \lecnot (page~175) provides the bound
\begin{align}
\begin{split}
\label{eq:kuninequalityAS2}
& \textstyle\sum_{i=1}^N (R_i A R_i^T D_i (\id-\xi_{0i}) (\id - \pi_i) R_i \mathbf{U},  D_i (\id-\xi_{0i}) (\id - \pi_i) R_i \mathbf{U})\\
&= \textstyle\sum_{i=1}^N \| R_i^T D_i (\id-\xi_{0i}) (\id - \pi_i) R_i \mathbf{U}\|_A^2 \le \tau \MC \, a(\mathbf{U}, \mathbf{U}).
\end{split}
\end{align}
The remaining term in $\tilde{b}(\mathcal{U},\mathcal{U})$ corresponds to the coarse operator $\tilde{E}$ where, using the bound \eqref{eq:tildeEBoundedByE} of \Cref{lemma:EandtildeE}, we have
\begin{align*}
(\tilde{E} {\mathbf U}_0,{\mathbf U}_0) &\le \lambda_{max}(E^{-1} \tilde{E}) \| Z\mathbf{U}_0 \|_{A}^{2} = \lambda_{max}(E^{-1} \tilde{E}) \|F\mathbf{U}\|_{A}^2.
\end{align*}
Now from \eqref{eq:FU}, letting $\delta>0$ be a positive parameter, making use of the Cauchy--Schwarz inequality, Young's inequality (with parameter $\delta$), Lemma~7.9 of \lecnot (page~171) and \eqref{eq:kuninequalityAS2} gives
\begin{align*}
\|F\mathbf{U}\|_{A}^2 &\le \| P_0 \mathbf{U} \|_{A}^{2} + 2a(P_0 \mathbf{U},(\tilde{P}_0-P_0) \textstyle\sum_{i=1}^{N} R_i^T D_i (\id-\xi_{0i}) (\id-\pi_i) R_i \mathbf{U})\\
& \mathrel{\phantom{=}} \mathrel+ \| (\tilde{P}_0-P_0) \textstyle\sum_{i=1}^{N} R_i^T D_i (\id-\xi_{0i}) (\id-\pi_i) R_i \mathbf{U} \|_{A}^{2}\\
&\le (1+\delta)\| P_0 \mathbf{U} \|_{A}^{2} + (1+\tfrac{1}{\delta})\| (\tilde{P}_0-P_0) \textstyle\sum_{i=1}^{N} R_i^T D_i (\id-\xi_{0i}) (\id-\pi_i) R_i \mathbf{U} \|_{A}^{2}\\
&\le (1+\delta) a(\mathbf{U},\mathbf{U}) + (1+\tfrac{1}{\delta}) \epsilon_{A}^{2} \| \textstyle\sum_{i=1}^{N} R_i^T D_i (\id-\xi_{0i}) (\id-\pi_i) R_i \mathbf{U} \|_{A}^{2}\\
&\le (1+\delta) a(\mathbf{U},\mathbf{U}) + (1+\tfrac{1}{\delta}) \epsilon_{A}^{2} \NC \textstyle\sum_{i=1}^{N} \| R_i^T D_i (\id-\xi_{0i}) (\id-\pi_i) R_i \mathbf{U} \|_{A}^{2}\\
&\le \left( (1+\delta) + (1+\tfrac{1}{\delta}) \epsilon_{A}^{2} \NC \MC \tau \right) a(\mathbf{U},\mathbf{U}).
\end{align*}
We can minimise over the parameter $\delta$, yielding $\delta = \epsilon_{A}\sqrt{\NC\MC\tau}$, and thus
\begin{align*}
(\tilde{E} {\mathbf U}_0,{\mathbf U}_0) &\le \lambda_{max}(E^{-1} \tilde{E}) (1+\epsilon_{A}\sqrt{\NC\MC\tau})^2 a(\mathbf{U},\mathbf{U}).
\end{align*}
Combining this estimate with \eqref{eq:kuninequalityAS2} gives $\tilde{b}(\mathcal{U},\mathcal{U}) \le c_T^{-1} \, a(\mathbf{U}, \mathbf{U})$ where
\begin{align}
\nonumber
c_T &= \frac{1}{\MC\tau + \lambda_{max}(E^{-1} \tilde{E}) (1+\epsilon_{A}\sqrt{\NC\MC\tau})^2}\\
\label{eq:c_T_inexact_as}
&= \frac{\lambda_{min}(E\tilde{E}^{-1})}{(1+\epsilon_{A}\sqrt{\NC\MC\tau})^2 + \lambda_{min}(E\tilde{E}^{-1})\MC\tau}.
\end{align}
Thus we see that the constant in \eqref{eq:c_T_inexact_as} is given solely in terms of the constants $\NC$, $\MC$ and $\tau$, and the minimal and maximal eigenvalues of $E\tilde{E}^{-1}$.

Altogether, the fictitious space lemma provides the following spectral bounds
\begin{align*}
c_T \, a(\mathbf{U},\mathbf{U}) \le a\left( \tilde{M}_{AS,2}^{-1} A \mathbf{U}, \mathbf{U} \right) \le c_R \, a(\mathbf{U},\mathbf{U})
\end{align*}
for all $\mathbf{U} \in \HO = \R^{\#{\mathcal N}}$, where $c_R$ and $c_T$ are given by \eqref{eq:c_R_inexact_as} and \eqref{eq:c_T_inexact_as} respectively, and
\begin{align*}
\tilde{M}_{AS,2}^{-1} &= \widetilde{\RL}_{AS,2} \tilde{B}^{-1} \widetilde{\RL}_{AS,2}^*\\
&= Z \tilde{E}^{-1} Z^T + (\id-\tilde{P}_0) \sum_{i=1}^N R_i^T B_i^{-1} R_i (\id-\tilde{P}_0^T)\\
&= \tilde{P}_0 A^{-1} + (\id-\tilde{P}_0) \sum_{i=1}^N R_i^T (R_i A R_i^T)^{-1} R_i (\id-\tilde{P}_0^T).
\end{align*}

\section{A SORAS method}
\label{sec:soras_method}

In this section we address the symmetrised optimised restricted additive Schwarz (SORAS) method. Here the symmetric positive definite subdomain matrices $B_i$ stem from using e.g.~Neumann matrices or optimised transmission conditions rather than the Dirichlet matrices $R_i A R_i^T$ within the additive Schwarz method. In addition, the partition of unity matrices $D_i$ are used symmetrically within the method in order that we can apply the present theory. As with the additive Schwarz method of the previous section, we utilise a coarse space including $V_G$ to ameliorate the near-kernel $G$.

\subsection{The underlying SORAS method}
\label{sub:soras}

We now define the abstract framework for the SORAS preconditioner. Let $\HP$ be defined as
\begin{align*}
\HP := \R^{\#\mathcal{N}_G} \times \prod_{i=1}^N \Gw
\end{align*}
endowed with the Euclidean scalar product. For $\mathcal{U} = (\mathbf{U}_0, (\mathbf{U}_i)_{1 \le i \le N}) \in \HP$ and $\mathcal{V} = (\mathbf{V}_0, (\mathbf{V}_i)_{1 \le i \le N}) \in \HP$, with $\mathbf{U}_0,\mathbf{V}_0 \in \R^{\#\mathcal{N}_G}$ and $\mathbf{U}_i,\mathbf{V}_i \in \Gw$ for $1 \le i \le N$, define the bilinear form $b \colon \HP \times \HP \longrightarrow \R$ arising from the coarse operator $E$ and the local SPD matrices $(B_i)_{1\le i \le N}$ such that
\begin{align*}
(\mathcal{U},\mathcal{V}) \mapsto b(\mathcal{U},\mathcal{V}) := (E\mathbf{U}_0,\mathbf{V}_0) + \sum_{i=1}^{N} (B_i \mathbf{U}_i, \mathbf{V}_i).
\end{align*}
Further, we denote by $B \colon \HP \longrightarrow \HP$ the block diagonal operator such that $(B\mathcal{U},\mathcal{V}) = b(\mathcal{U},\mathcal{V})$ for all $\mathcal{U},\mathcal{V} \in \HP$. Now, using the $A$-orthogonal projection $P_0$ on $V_G$, for any $\mathcal{U} = (\mathbf{U}_0, (\mathbf{U}_i)_{1 \le i \le N}) \in \HP$ we define the linear operator $\RL_{SORAS} \colon \HP \longrightarrow \HO$ by
\begin{align*}
\RL_{SORAS}(\mathcal{U}) : = Z\mathbf{U}_0 + (\id-P_0) \sum_{i=1}^{N} R_i^T D_i \mathbf{U}_i.
\end{align*}

In order to apply the fictitious space lemma we check the three assumptions required.\\

$\bullet$ $\RL_{SORAS}$ is onto:\\
Let $\mathbf{U} \in \HO$, we have
\begin{align*}
\mathbf{U} &= P_0 \mathbf{U} + (\id - P_0) \mathbf{U} \\
&= P_0 \mathbf{U} + (\id - P_0) \textstyle\sum_{i=1}^N R_i^T D_i R_i \mathbf{U} \\
&= P_0 \mathbf{U} + (\id - P_0) \textstyle\sum_{i=1}^N R_i^T D_i (\id - \xi_{0i}) R_i \mathbf{U} + \underbrace{(\id - P_0) \textstyle\sum_{i=1}^N R_i^T D_i \xi_{0i} R_i \mathbf{U}}_{=0}.
\end{align*}
Let us consider the last equality: since $P_0 \mathbf{U} \in V_G$ there exists $\mathbf{U}_0$ such that $Z \mathbf{U}_0 = P_0 \mathbf{U}$, meanwhile the third term is zero since $\sum_{i=1}^N R_i^T D_i \xi_{0i} R_i \mathbf{U}\in V_G$. Note also that $(\id - \xi_{0i}) R_i \mathbf{U} \in \Gw$. Therefore, we have
\begin{align}
\label{eq:RontoOnelevelSORAS}
\mathbf{U} = \RL_{SORAS}(\mathbf{U}_0, ((\id-\xi_{0i}) R_i \mathbf{U})_{1\leq i\leq N}).
\end{align}

$\bullet$ Continuity of $\RL_{SORAS}$:\\
We have to estimate a constant $c_R$ such that for all $\mathcal{U} = (\mathbf{U}_0, (\mathbf{U}_i)_{1 \le i \le N}) \in \HP$ we have
\begin{align*}
a( \RL_{SORAS} (\mathcal{U}), \RL_{SORAS} (\mathcal{U})) &\le c_R \, b(\mathcal{U}, \mathcal{U}) = c_R \, [ (E\mathbf{U}_0, \mathbf{U}_0) + \sum_{i=1}^N (B_i \mathbf{U}_i, \mathbf{U}_i)].
\end{align*}
Note that $\mathbf{U}_i \in \Gw$ here for $1 \le i \le N$. Now we have the following estimate using the $A$-orthogonality of $\id-P_0$ and Lemma~7.9 in \lecnot (page~171):
\begin{align*}
a(\RL_{AS}(\mathcal{U}), \RL_{AS}(\mathcal{U})) &= \| Z \mathbf{U}_0 + (\id-P_0) \textstyle\sum_{i=1}^N R_i^T D_i \mathbf{U}_i \|_A^2\\
&= \| Z \mathbf{U}_0 \|_A^2 + \| (\id-P_0) \textstyle\sum_{i=1}^N R_i^T D_i \mathbf{U}_i \|_A^2\\
&\le (E \mathbf{U}_0, \mathbf{U}_0) + \NC \textstyle\sum_{i=1}^N \| R_i^T D_i \mathbf{U}_i \|_A^2\\
&\le (E \mathbf{U}_0, \mathbf{U}_0) + \NC \gamma_0 \textstyle\sum_{i=1}^N (B_i \mathbf{U}_i, \mathbf{U}_i),
\end{align*}
where
\begin{align}
\label{eq:gammaOneLevelSORAS}
\gamma_0 := \max_{1 \le i \le N} \max_{\mathbf{V}_i \in \Gw} \frac{\| R_i^T D_i \mathbf{V}_i \|_A^2}{(B_i \mathbf{V}_i, \mathbf{V}_i)} = \max_{1 \le i \le N} \max_{\mathbf{V}_i \in \Gw} \frac{(D_i R_i A R_i^T D_i \mathbf{V}_i, \mathbf{V}_i)}{(B_i \mathbf{V}_i, \mathbf{V}_i)}.
\end{align}
Thus the estimate of the constant of continuity of $\RL_{AS}$ can be chosen as
\begin{align*}
c_R := \max(1,\NC \gamma_0).
\end{align*}

$\bullet$ Stable decomposition with $\RL_{SORAS}$:\\
Let $\mathbf{U} \in \HO$ be decomposed as in \eqref{eq:RontoOnelevelSORAS}. Determining the stability of the decomposition consists in estimating a constant $c_T > 0$ such that
\begin{align*}
c_T \, [(E\mathbf{U}_0, \mathbf{U}_0) + \sum_{j=1}^N (B_j (\id-\xi_{0j}) R_j \mathbf{U},  (\id-\xi_{0j}) R_j \mathbf{U})] \leq a(\mathbf{U}, \mathbf{U}).
\end{align*}
We have
\begin{align}
\nonumber
\textstyle\sum_{j=1}^N (B_j (\id-\xi_{0j}) R_j \mathbf{U}, (\id-\xi_{0j}) R_j \mathbf{U}) &\le \tau_0 \textstyle\sum_{j=1}^N (A_j^{Neu} R_j \mathbf{U}, R_j \mathbf{U})\\
\label{eq:estimateSORAS}
&\le \tau_0 \MC \, a(\mathbf{U}, \mathbf{U}),
\end{align}
where
\begin{align}
\label{eq:tauOneLevelSORAS}
\tau_0 := \max_{1 \le j \le N} \max_{\mathbf{V}_j \in \R^{\#\mathcal{N}_j}} \frac{(B_j (\id-\xi_{0j}) \mathbf{V}_j, (\id-\xi_{0j}) \mathbf{V}_j)}{(A_j^{Neu} \mathbf{V}_j, \mathbf{V}_j)},
\end{align}
and in the second step we have used Lemma~7.13 in \lecnot (page~ 175). By applying \eqref{eq:estimateSORAS}, we obtain
\begin{align*}
& (E\mathbf{U}_0, \mathbf{U}_0) + \textstyle\sum_{j=1}^N (B_j (\id-\xi_{0j}) R_j \mathbf{U}, (\id-\xi_{0j}) R_j \mathbf{U})\\
&\leq a(P_0\mathbf{U}, P_0\mathbf{U}) + \tau_0 \MC \, a(\mathbf{U}, \mathbf{U})\\
&\le (1 + \MC \tau_0) \, a(\mathbf{U}, \mathbf{U}).
\end{align*}
Thus we can take
\begin{align*}
c_T := \frac{1}{1 + \MC \tau_0}.
\end{align*}
Note that the choice $B_j := A_j^{Neu}$ would give $\tau_0 = 1$. The condition number estimate given by the fictitious space lemma for the induced preconditioner is
\begin{align*}
\kappa(M_{SORAS}^{-1} A) \le (1 + \MC \tau_0) \max(1, \NC \gamma_0).
\end{align*}
The development that now follows is motivated by the fact that the numbers $\gamma_0$ and $\tau_0$ may be very large due to the shape and size of the domain or the heterogeneities in the coefficients of the problem. This leads to a bad condition number estimate. The fix is to enlarge the coarse space by GEVPs induced by the formulae \eqref{eq:gammaOneLevelSORAS} and \eqref{eq:tauOneLevelSORAS}. More precisely, we introduce in the next subsection two generalized eigenvalue problems and the related GenEO coarse space.

\subsection{SORAS with GenEO}
\label{sub:soras_geneo}

\begin{gevp}[Generalized Eigenvalue Problem for the lower bound]
\label{gevp:tauthresholdSORAS}
For each subdomain $1 \le j \le N$, we introduce the generalized eigenvalue problem
\begin{align}
\label{eq:eigAtildeBSORAS}
\begin{array}{c}
\text{Find } (\mathbf{V}_{jk},\lambda_{jk}) \in \R^{\#\mathcal{N}_j} \setminus \{0\} \times \R
\mbox{ such that}\\[1ex]
(\id-\xi_{0j}^T) B_j (\id-\xi_{0j}) \mathbf{V}_{jk} = \lambda_{jk} A_j^{Neu} \mathbf{V}_{jk}.
\end{array}
\end{align}
Let $\tau>0$ be a user-defined threshold, we define $V_{j,geneo}^\tau \subset \R^{\#\mathcal{N}}$ as the vector space spanned by the family of vectors $(R_j^T D_j (\id - \xi_{0j}) \mathbf{V}_{jk})_{\lambda_{jk} > \tau}$ corresponding to eigenvalues larger than $\tau$. Let $V_{geneo}^\tau$ be the vector space spanned by the collection over all subdomains of vector spaces $(V_{j,geneo}^\tau)_{1 \le j \le N}$.
\end{gevp}

In the theory that follows for the stable decomposition estimate but not in the algorithm itself, we will make use of $\pi_j$ defined as the projection from $R^{\#\mathcal{N}_j}$ on $V_{j,\tau} := \text{Span}\{ \mathbf{V}_{jk} |\, \lambda_{jk} > \tau \}$ parallel to $\text{Span}\{ \mathbf{V}_{jk} |\, \lambda_{jk} \le \tau \}$. The key to \Cref{gevp:tauthresholdSORAS} is the following bound, derived from Lemma~7.7 in \cite{Dolean:2015:IDDSiam} (page~168): for all $\mathbf{U}_j \in R^{\#\mathcal{N}_j}$
\begin{align}
\label{eq:gevpLowerBoundSORAS}
(B_j (\id-\xi_{0j}) (\id-\pi_j) \mathbf{U}_j, (\id-\xi_{0j}) (\id-\pi_j) \mathbf{U}_j) &\le \tau (A_j^{Neu} \mathbf{U}_j, \mathbf{U}_j).
\end{align}

\begin{gevp}[Generalized Eigenvalue Problem for the upper bound]
\label{gevp:gammathresholdSORAS}
For each subdomain $1 \le i \le N$, we introduce the generalized eigenvalue problem
\begin{align}
\label{eq:eigDRARtDBSORAS}
\begin{array}{c}
\text{Find } (\mathbf{U}_{ik},\mu_{ik}) \in \Gw \setminus \{0\} \times \R
\mbox{ such that}\\[1ex]
D_i R_i A R_i^T D_i \mathbf{U}_{ik} = \mu_{ik} \Bw \mathbf{U}_{ik}.
\end{array}
\end{align}
Let $\gamma>0$ be a user-defined threshold, we define $V_{i,geneo}^\gamma \subset \Gw$ as the vector space spanned by the family of vectors $(R_i^T D_i \mathbf{U}_{ik})_{\mu_{ik} > \gamma}$ corresponding to eigenvalues larger than $\gamma$. Let $V_{geneo}^\gamma$ be the vector space spanned by the collection over all subdomains of vector spaces $(V_{i,geneo}^\gamma)_{1 \le i \le N}$.
\end{gevp}

In our theory for the continuity estimate but not in the algorithm itself, we will make use of the projection $\eta_i$ from $\Gw$ on $V_{i,\gamma} := \mathrm{span}\{\mathbf{U}_{ik} |\, \mu_{ik} > \gamma\}$ parallel to $\mathrm{span}\{\mathbf{U}_{ik} |\, \mu_{ik} \le \gamma\}$. The key to \Cref{gevp:gammathresholdSORAS} is the following bound, derived from Lemma~7.7 in \cite{Dolean:2015:IDDSiam} (page~168): for all $\mathbf{U}_i \in \Gw$
\begin{align}
\label{eq:gevpUpperBoundSORAS}
(A R_i^T D_i (\id - \eta_i) \mathbf{U}_i, R_i^T D_i (\id - \eta_{i}) \mathbf{U}_i) &\le \gamma (\Bw \mathbf{U}_i, \mathbf{U}_i) = \gamma (B_i \mathbf{U}_i, \mathbf{U}_i).
\end{align}

We can now build the coarse space $V_0$ from the near-kernel $G$ along with \Cref{gevp:tauthresholdSORAS,gevp:gammathresholdSORAS}, defining the following vector space sum:
\begin{align}
\label{eq:V0GenEOSORAS}
V_0 := V_G + V_{geneo}^\tau + V_{geneo}^\gamma.
\end{align}
The coarse space $V_0$ is spanned by the columns of a full rank rectangular matrix with $\#\mathcal{N}_0$ columns, which we will denote by $Z$.

We can now define the abstract framework for the SORAS with GenEO preconditioner. Let $\HP$ be defined by
\begin{align*}
\HP := \R^{\#\mathcal{N}_0} \times \prod_{i=1}^N \Gw
\end{align*}
endowed with the Euclidean scalar product. We now define the bilinear form $b \colon \HP \times \HP \longrightarrow \R$ arising from the coarse operator $E = Z^T A Z$ and the local SPD matrices $(B_i)_{1\le i \le N}$ such that
\begin{align*}
(\mathcal{U},\mathcal{V}) \mapsto b(\mathcal{U},\mathcal{V}) := (E\mathbf{U}_0,\mathbf{V}_0) + \sum_{i=1}^{N} (B_i \mathbf{U}_i, \mathbf{V}_i).
\end{align*}
Further, we denote by $B \colon \HP \longrightarrow \HP$ the block diagonal operator such that $(B\mathcal{U},\mathcal{V}) = b(\mathcal{U},\mathcal{V})$ for all $\mathcal{U},\mathcal{V} \in \HP$. Using the $A$-orthogonal projection $P_0$ on $V_0$, we define $\RL_{SORAS,2} \colon \HP \longrightarrow \HO$ as
\begin{align*}
\RL_{SORAS,2}(\mathcal{U}) : = Z\mathbf{U}_0 + (\id-P_0) \sum_{i=1}^{N} R_i^T D_i \mathbf{U}_i.
\end{align*}

In order to apply the fictitious space lemma we check the three assumptions as before. Surjectivity of $\RL_{SORAS,2}$ follows identically to that for $\RL_{SORAS}$ as in \Cref{sub:soras}. Thus we need only consider continuity and the stable decomposition.

$\bullet$ Continuity of $\RL_{SORAS,2}$:\\
For $\mathcal{U} = (\mathbf{U}_0, (\mathbf{U}_i)_{1 \le i \le N}) \in \HP$, by $A$-orthogonality of $I-P_0$, using Lemma~7.9 in \lecnot (page~171) along with the GEVP bound \eqref{eq:gevpUpperBoundSORAS} we have
\begin{align*}
a(\RL(\mathcal{U}),\RL(\mathcal{U})) &= \| Z \mathbf{U}_0 + (\id-P_0) \textstyle\sum_{i=1}^N R_i^T D_i \mathbf{U}_i \|_{A}^{2}\\
&= \| Z \mathbf{U}_0 \|_{A}^{2} + \| (\id-P_0) \textstyle\sum_{i=1}^N R_i^T D_i \mathbf{U}_i \|_{A}^{2}\\
&= (E \mathbf{U}_0, \mathbf{U}_0) + \| (\id-P_0) \textstyle\sum_{i=1}^N R_i^T D_i (\id-\eta_{i}) \mathbf{U}_i \|_{A}^{2}\\
&\le (E \mathbf{U}_0,\mathbf{U}_0) + (A \textstyle\sum_{i=1}^N R_i^T D_i (\id-\eta_{i}) \mathbf{U}_i, \textstyle\sum_{i=1}^N R_i^T D_i (\id-\eta_{i}) \mathbf{U}_i)\\
&\le (E \mathbf{U}_0, \mathbf{U}_0) + \NC \textstyle\sum_{i=1}^N (A R_i^T D_i (\id-\eta_{i}) \mathbf{U}_i, R_i^T D_i (\id-\eta_{i}) \mathbf{U}_i)\\
&\le (E \mathbf{U}_0, \mathbf{U}_0) + \NC\gamma \textstyle\sum_{i=1}^N (B_i \mathbf{U}_i, \mathbf{U}_i).
\end{align*}
Thus the estimate of the constant of continuity of $\RL_{SORAS,2}$ can be chosen as
\begin{align*}
c_R := \max(1,\NC\gamma).
\end{align*}

$\bullet$ Stable decomposition with $\RL_{SORAS,2}$:\\
Let $\mathbf{U} \in \HO$, we have
\begin{align*}
\mathbf{U} &= P_0 \mathbf{U} + (\id - P_0) \textstyle\sum_{i=1}^N R_i^T D_i R_i \mathbf{U} \\
&= P_0 \mathbf{U} + (\id - P_0) \textstyle\sum_{i=1}^N R_i^T D_i (\id - \xi_{0i}) R_i \mathbf{U} + \underbrace{(\id - P_0) \textstyle\sum_{i=1}^N R_i^T D_i \xi_{0i} R_i \mathbf{U}}_{=0}\\
&= P_0 \mathbf{U} + (\id - P_0) \textstyle\sum_{i=1}^N R_i^T D_i (\id - \xi_{0i}) (\id - \pi_i) R_i \mathbf{U}\\
& \mathrel{\phantom{=}} \mathrel+ \underbrace{(\id - P_0) \textstyle\sum_{i=1}^N R_i^T D_i (\id - \xi_{0i}) \pi_i R_i \mathbf{U}}_{=0}.
\end{align*}
The very last term is zero since for all $i$, $R_i^T D_i (\id - \xi_{0i}) \pi_i R_i \mathbf{U} \in V_{geneo}^\tau \subset V_0$. Let $\mathbf{U}_0 \in \R^{\#\mathcal{N}_0}$ be such that $Z \mathbf{U}_0 = P_0\mathbf{U}$, then we can choose the decomposition
\begin{align*}
\mathbf{U} = \RL_{SORAS,2}(\mathbf{U}_0, ((\id - \xi_{0i}) (\id - \pi_i) R_i \mathbf{U})_{1 \leq i \leq N}).
\end{align*}
With this decomposition, using the GEVP bound \eqref{eq:gevpLowerBoundSORAS} and Lemma~7.13 in \lecnot (page~175), we have
\begin{align*}
& (E\mathbf{U}_0, \mathbf{U}_0) + \textstyle\sum_{j=1}^N (B_j (\id-\xi_{0j}) (\id - \pi_j) R_j \mathbf{U}, (\id-\xi_{0j}) (\id - \pi_j) R_j \mathbf{U})\\
&\le a(Z\mathbf{U}_0, Z\mathbf{U}_0) + \tau \textstyle\sum_{j=1}^N (A_j^{Neu} R_j \mathbf{U}, R_j \mathbf{U})\\
&\le a(P_0\mathbf{U}, P_0\mathbf{U}) + \tau \MC \, a(\mathbf{U}, \mathbf{U})\\
&\le (1 + \MC \tau) \, a(\mathbf{U}, \mathbf{U}).
\end{align*}
Thus the stable decomposition property holds with constant
\begin{align*}
c_T := \frac{1}{1 + \MC \tau}.
\end{align*}

Altogether, for given user-defined positive constants $\tau$ and $\gamma$, the SORAS with GenEO preconditioner yields the condition number estimate
\begin{align*}
\kappa(M^{-1}_{SORAS,2} A) \le (1 + \MC \tau) \max(1,\NC\gamma),
\end{align*}
where
\begin{align*}
M^{-1}_{SORAS,2} &= \RL_{SORAS,2} B^{-1} \RL_{SORAS,2}^*.
\end{align*}
In order to specify this preconditioner we need the adjoint operator $\RL_{SORAS,2}^*$ which now requires an additional projection. Let $q_i$ denote the ($l_2$-)orthogonal projection from $\R^{\#{\mathcal N}_i}$ on $G_i^{\bot_{B_i}}$, then the adjoint is defined by the relationship
\begin{align*}
(\mathcal{U}, \RL_{SORAS,2}^*(\mathbf{V}))_{\HP} &= (\RL_{SORAS,2}(\mathcal{U}), \mathbf{V})_{\HO}\\
&= (Z \mathbf{U}_0, \mathbf{V}) + \left((\id-P_0) \textstyle\sum_{i=1}^{N} R_i^T D_i \mathbf{U}_i, \mathbf{V}\right)\\
&= (\mathbf{U}_0, Z^T \mathbf{V}) + \textstyle\sum_{i=1}^{N} \left(\mathbf{U}_i, D_i R_i (\id-P_0^T) \mathbf{V}\right)\\
&= (\mathbf{U}_0, Z^T \mathbf{V}) + \textstyle\sum_{i=1}^{N} \left(\mathbf{U}_i, q_i D_i R_i (\id-P_0^T) \mathbf{V}\right)
\end{align*}
for all $\mathcal{U} = (\mathbf{U}_0, (\mathbf{U}_i)_{1 \le i \le N}) \in \HP$ and $\mathbf{V} \in \HO$. Note that the projection $q_i$ ensures that $q_i D_i R_i (\id-P_0^T) \mathbf{V} \in G_i^{\bot_{B_i}}$ while leaving the ($l_2$-)inner product unchanged as $\mathbf{U}_i \in G_i^{\bot_{B_i}}$. Hence we identify that
\begin{align*}
\RL_{SORAS,2}^*(\mathbf{V}) = (Z^T \mathbf{V},(q_i D_i R_i (\id-P_0^T) \mathbf{V})_{1 \le i \le N})
\end{align*}
and thus
\begin{align*}
M_{SORAS,2}^{-1} &= Z (Z^T A Z)^{-1} Z^T + (\id-P_0) \sum_{i=1}^N R_i^T D_i B_i^\dag q_i D_i R_i (\id-P_0^T).
\end{align*}

To determine an explicit form for $q_i$, suppose we wish to apply the projection to $\mathbf{U}_i \in \R^{\#{\mathcal N}_i}$, then $q_i \mathbf{U}_i$ satisfies the constrained minimisation problem\footnote{Here, by abuse of notation, $G_i$ represents a matrix whose columns form a basis for the near-kernel space $G_i$.}
\begin{align}
\label{eq:OptProbq}
\min_{\mathbf{V}_i \in G_i^{\bot_{B_i}}} \frac{1}{2} \|\mathbf{V}_i - \mathbf{U}_i\|^2 = \min_{\mathbf{V}_i \, | \, G_i^T B_i \mathbf{V}_i = \mathbf{0}} \frac{1}{2} \|\mathbf{V}_i - \mathbf{U}_i\|^2.
\end{align}
We can solve the optimisation problem \eqref{eq:OptProbq} using the Lagrange multiplier method. Introducing the multipliers $\boldsymbol{\lambda} \in \R^{\#{\mathcal N}_{G_i}}$, the optimality conditions for $\mathbf{V}_i = q_i \mathbf{U}_i$ are given by
\begin{align*}
(q_i \mathbf{U}_i - \mathbf{U}_i) - B_i G_i \boldsymbol{\lambda} &= \mathbf{0},\\
G_i^T B_i q_i \mathbf{U}_i &= \mathbf{0}.
\end{align*}
Solving for $\boldsymbol{\lambda}$ yields
\begin{align*}
\boldsymbol{\lambda} = - (G_i^T B_i^2 G_i)^{-1} G_i^T B_i \mathbf{U}_i,
\end{align*}
and thus, as $\mathbf{U}_i \in \R^{\#{\mathcal N}_i}$ was arbitrary, we can derive the explicit form
\begin{align*}
q_i = \id - B_i G_i (G_i^T B_i^2 G_i)^{-1} G_i^T B_i.
\end{align*}
We also require an explicit expression for $B_i^\dag = (\id-\xi_{0i})B_i^{-1}$. Since $\xi_{0i}$ is the $B_i$-orthogonal projection from $\R^{\#{\mathcal N}_i}$ on $G_i$ parallel to $G_i^{\bot_{B_i}}$ it is given by
\begin{align*}
\xi_{0i} = G_i (G_i^T B_i G_i)^{-1} G_i^T B_i
\end{align*}
and thus we have
\begin{align*}
(\id-\xi_{0i}) B_i^{-1} q_i &= (B_i^{-1} - G_i (G_i^T B_i G_i)^{-1} G_i^T)(I - B_i G_i (G_i^T B_i^2 G_i)^{-1} G_i^T B_i)\\
&= B_i^{-1} - G_i (G_i^T B_i G_i)^{-1} G_i^T\\
&\mathrel{\phantom{=}} \mathrel- G_i (G_i^T B_i^2 G_i)^{-1} G_i^T B_i + G_i (G_i^T B_i^2 G_i)^{-1} G_i^T B_i\\
&= B_i^{-1} - G_i (G_i^T B_i G_i)^{-1} G_i^T\\
&= (\id-\xi_{0i}) B_i^{-1}.
\end{align*}
From the penultimate expression we see that we have symmetry of $B_i^\dag q_i$ and, moreover, that the inclusion of $q_i$ is, in theory, unnecessary since $B_i^\dag q_i = B_i^\dag$. Hence, we find that the two-level preconditioner can be given by
\begin{align*}
M_{SORAS,2}^{-1} &= Z (Z^T A Z)^{-1} Z^T + (\id-P_0) \sum_{i=1}^N R_i^T D_i (\id-\xi_{0i}) B_i^{-1} D_i R_i (\id-P_0^T).
\end{align*}
Further, note that we have $(\id-P_0) R_i^T D_i (\id-\xi_{0i}) \mathbf{U}_i = (\id-P_0) R_i^T D_i \mathbf{U}_i$ for any $\mathbf{U}_i \in \R^{\#{\mathcal N}_i}$ since $ R_i^T D_i \xi_{0i} \mathbf{U}_i \in V_G \subset V_0$. Additionally, note that the $A$-orthogonal projection $P_0$ on $V_0$ is given by
\begin{align*}
P_0 = Z (Z^T A Z)^{-1} Z^T A.
\end{align*}
Hence, we deduce that we can write the SORAS with GenEO preconditioner in a simpler expression as
\begin{align}
\label{eq:SORAS-GenEO-preconditioner}
M_{SORAS,2}^{-1} &= P_0 A^{-1} + (\id-P_0) \sum_{i=1}^N R_i^T D_i B_i^{-1} D_i R_i (\id-P_0^T).
\end{align}

\subsection{Inexact coarse solves}
\label{sub:soras_geneo_with_inexact_coarse_solves}

We now consider the SORAS with GenEO method with inexact coarse solves given by $\tilde{E}$ satisfying Assumption \ref{ass:spd_tildeE}. We follow the same premise as \Cref{sub:as_geneo_with_inexact_coarse_solves} and again let $\tilde{P}_0 := Z \tilde{E}^{-1} Z^T A$ be the inexact coarse solve approximation of $P_0 = Z E^{-1} Z^T A$. Our analysis predominantly follows that in \cite{Nataf:2020:MAR}.

First we revisit the eigenvalue problems in \Cref{gevp:tauthresholdSORAS,gevp:gammathresholdSORAS}. Recall that $\pi_j$ is the projection from $\R^{\#{\mathcal N}_j}$ on $V_{j,\tau} := \text{Span}\{ \mathbf{V}_{jk} |\, \lambda_{jk} > \tau \}$ parallel to $\text{Span}\{ \mathbf{V}_{jk} |\, \lambda_{jk} \le \tau \}$, where $(\mathbf{V}_{jk},\lambda_{jk})$ are eigenpairs from \Cref{gevp:tauthresholdSORAS}. Similarly, $\eta_i$ is the $B_i$-orthogonal projection from $G_i^{\bot_{B_i}}$ on
\begin{align*}
V_{i,\gamma} := \mathrm{span}\{\mathbf{U}_{ik} |\, \mu_{ik} > \gamma\}
\end{align*}
parallel to
\begin{align*}
W_{i,\gamma} := \mathrm{span}\{\mathbf{U}_{ik} |\, \mu_{ik} \le \gamma\},
\end{align*}
where $({\mathbf U}_{ik},\mu_{ik})$ are eigenpairs from \Cref{gevp:gammathresholdSORAS}. Now, for $1 \le j \le N$, let us define the $(\id-\xi_{0j}^T) B_j (\id-\xi_{0j})$-orthogonal projection $p_j$ from $R^{\#{\mathcal N}_j}$ on
\begin{align*}
V_{j,\tau,\gamma} := V_{j,\tau} + V_{j,\gamma}.
\end{align*}
Note that $V_{j,\gamma} \subset G_j^{\bot_{B_j}}$ and $G_j \nsubseteq V_{j,\tau}$ since any vector in $G_j$ corresponds to a zero eigenvalue of \Cref{gevp:tauthresholdSORAS}. Thus $G_j \cap V_{j,\tau,\gamma} = \left\lbrace0\right\rbrace$ and the projection is well-defined and, letting $Y$ be a basis of $V_{j,\tau,\gamma}$, given by the formula
\begin{align*}
p_j = Y \left( Y^{T} (\id-\xi_{0j}^T) B_j (\id-\xi_{0j}) Y \right)^{-1} Y^{T} (\id-\xi_{0j}^T) B_j (\id-\xi_{0j}).
\end{align*}
While $(\id-\xi_{0j}^T) B_j (\id-\xi_{0j})$ is singular on $G_j$, it is nonsingular on $\mathrm{range}(Y) = V_{j,\tau,\gamma}$ and thus we can take the required inverse.

We now show that suitable analogous results to Lemma 5 in \cite{Nataf:2020:MAR} for $p_j$ hold in our case. Owing to \Cref{gevp:tauthresholdSORAS}, we have that for all $\mathbf{U}_j \in R^{\#{\mathcal N}_j}$
\begin{align*}
& \tau (A_j^{Neu} \mathbf{U}_j, \mathbf{U}_j)\\
&\ge (B_j (\id-\xi_{0j}) (\id-\pi_{j}) \mathbf{U}_j, (\id-\xi_{0j}) (\id-\pi_{j}) \mathbf{U}_j)\\
&= (B_j (\id-\xi_{0j}) (\id-p_j+(p_j-\pi_{j})) \mathbf{U}_j, (\id-\xi_{0j}) (\id-p_j+(p_j-\pi_{j})) \mathbf{U}_j)\\
&= \|(\id-\xi_{0j}) (\id-p_j) \mathbf{U}_j\|_{B_j}^2 + \|(\id-\xi_{0j}) (p_j-\pi_j) \mathbf{U}_j\|_{B_j}^2\\
& \mathrel{\phantom{=}} \mathrel+ 2 \, (B_j (\id-\xi_{0j}) (\id-p_j) \mathbf{U}_j, (\id-\xi_{0j}) (p_j-\pi_{j}) \mathbf{U}_j)\\
&= \|(\id-\xi_{0j}) (\id-p_j) \mathbf{U}_j\|_{B_j}^2 + \|(\id-\xi_{0j}) (p_j-\pi_j) \mathbf{U}_j\|_{B_j}^2,
\end{align*}
the cross term in the penultimate line being zero due to the fact that $p_j$ is the $(\id-\xi_{0j}^T) B_j (\id-\xi_{0j})$-orthogonal projection and noting that $\pi_j$ also projects into $V_{j,\tau,\gamma}$. Thus we deduce that for all $\mathbf{U}_j \in R^{\#{\mathcal N}_j}$
\begin{align*}
(B_j (\id-\xi_{0j}) (\id-p_j) \mathbf{U}_j, (\id-\xi_{0j}) (\id-p_j) \mathbf{U}_j) \le \tau (A_j^{Neu} \mathbf{U}_j, \mathbf{U}_j)
\end{align*}
and, moreover, letting $\mathbf{U}_{j} = R_j \mathbf{U}$, summing over $j$ and using Lemma~7.13 of \cite{Dolean:2015:IDDSiam} (page~175), for all $\mathbf{U} \in R^{\#{\mathcal N}}$ we have
\begin{align}
\label{eq:lowerBoundFromEig-SORAS-special}
\sum_{j=1}^{N} (B_j (\id-\xi_{0j}) (\id-p_j) R_j \mathbf{U}, (\id-\xi_{0j}) (\id-p_j) R_j \mathbf{U}) \le \MC \tau \, a(\mathbf{U},\mathbf{U}).
\end{align}

In defining the preconditioner, we will want to operate within the space $W_{i,\gamma}$. Let $b_{W_{i,\gamma}}$ denote the restriction of $b_i$ to $W_{i,\gamma} \times W_{i,\gamma}$ so that
\begin{align*}
b_{W_{i,\gamma}} &\colon W_{i,\gamma} \times W_{i,\gamma} \longrightarrow \R, & (\mathbf{U}_i, \mathbf{V}_i) &\mapsto b_i(\mathbf{U}_i, \mathbf{V}_i).
\end{align*}
The Riesz representation theorem gives the existence of a unique isomorphism $B_{W_{i,\gamma}} \colon W_{i,\gamma} \longrightarrow W_{i,\gamma}$ into itself so that for all $\mathbf{U}_i, \mathbf{V}_i \in W_{i,\gamma}$ we have
\begin{align*}
b_{W_{i,\gamma}}(\mathbf{U}_i,\mathbf{V}_i) = (B_{W_{i,\gamma}} \mathbf{U}_i,\mathbf{V}_i).
\end{align*}
The inverse of $B_{W_{i,\gamma}}$ will be denoted by $\tilde{B}_i^\dag$ and is given by the following formula
\begin{equation}
\tilde{B}_i^\dag = (\id - \eta_i)(\id - \xi_{0i}) B_i^{-1}.
\end{equation}
In order to check this formula, we have to show that
\begin{align}
\label{eq:InverseB_WFormula}
B_{W_{i,\gamma}} (\id - \eta_i) (\id - \xi_{0i})B_i^{-1} y = y
\end{align}
for all $y \in W_{i,\gamma}$. Let $z \in W_{i,\gamma}$, using the fact that $(\id - \eta_i) (\id - \xi_{0i})$ is the $b_i$-orthogonal projection on $W_{i,\gamma}$, we have
\begin{align*}
(B_{W_{i,\gamma}} (\id - \eta_i) (\id - \xi_{0i}) B_i^{-1} y, z) &= b_i((\id - \eta_i) (\id - \xi_{0i}) B_i^{-1} y, z)\\
&= b_i(B_i^{-1} y, (\id - \xi_{0i}) (\id - \eta_i) z)\\
&= b_i(B_i^{-1} y, z) = (y,z).
\end{align*}
Since this equality holds for any $z \in W_{i,\gamma}$, this proves that \eqref{eq:InverseB_WFormula} holds and thus that $\tilde{B}_i^\dag$ provides the inverse of $B_{W_{i,\gamma}}$.

We now define the abstract framework for the preconditioner. Let $\HP$ be defined by
\begin{align*}
\HP := \R^{\#{\mathcal N}_0} \times \Pi_{i=1}^N W_{i,\gamma}
\end{align*}
endowed with standard Euclidean scalar product. We make use of the bilinear form $\tilde{b} \colon \HP \times \HP \longrightarrow \R$, with block diagonal matrix form $\tilde{B}$, such that
\begin{align*}
\tilde{b}(\mathcal{U},\mathcal{V}) := (\tilde{E}\mathbf{U}_0,\mathbf{V}_0) + \sum_{i=1}^{N} (B_i \mathbf{U}_i, \mathbf{V}_i),
\end{align*}
and consider the linear operator $\widetilde{\RL}_{SORAS,2} \colon \HP \longrightarrow \HO$ defined by
\begin{align*}
\widetilde{\RL}_{SORAS,2}(\mathcal{U}) : = Z\mathbf{U}_0 + (\id-\tilde{P}_0) \sum_{i=1}^{N} R_i^T D_i \mathbf{U}_i.
\end{align*}
Note that
\begin{align*}
\widetilde{\RL}_{SORAS,2}(\mathcal{U}) - \RL_{SORAS,2}(\mathcal{U}) = (P_0-\tilde{P}_0) \sum_{i=1}^{N} R_i^T D_i \mathbf{U}_i \in V_0
\end{align*}
since $\mathrm{Im}(P_0-\tilde{P}_0) \subset  V_0$ .

Following similar arguments to those presented in \cite{Nataf:2020:MAR}, we now check the three assumptions of the fictitious space lemma.

$\bullet$ $\widetilde{\RL}_{SORAS,2}$ is onto:\\
For $\mathbf{U} \in \HO$ we have that
\begin{align*}
\mathbf{U} &= P_0 \mathbf{U} + (\id - P_0) \mathbf{U} = P_0 \mathbf{U} + (\id-P_0) \textstyle\sum_{i=1}^{N} R_i^T D_i R_i \mathbf{U}\\
&= P_0 \mathbf{U} + (\id-P_0) \textstyle\sum_{i=1}^{N} R_i^T D_i (\id-\xi_{0i}) R_i \mathbf{U}\\
&= P_0 \mathbf{U} + (\id-P_0) \textstyle\sum_{i=1}^{N} R_i^T D_i (\id-\eta_i) (\id-\xi_{0i}) R_i \mathbf{U}\\
&= F_0 \mathbf{U} + (\id-\tilde{P}_0) \textstyle\sum_{i=1}^{N} R_i^T D_i (\id-\eta_i) (\id-\xi_{0i}) R_i \mathbf{U},
\end{align*}
where
\begin{align*}
F_0 \mathbf{U} = P_0 \mathbf{U} + (\tilde{P}_0-P_0) \textstyle\sum_{i=1}^{N} R_i^T D_i (\id-\eta_i) (\id-\xi_{0i}) R_i \mathbf{U} \in V_0.
\end{align*}
Let $\mathbf{U}_0 \in \R^{\#{\mathcal N}_0}$ be such that $Z \mathbf{U}_0 = F_0\mathbf{U}$, then we have the decomposition
\begin{align*}
\mathbf{U} = \widetilde{\RL}_{SORAS,2}(\mathbf{U}_0,((\id-\eta_i) (\id-\xi_{0i}) R_i \mathbf{U})_{1 \le i \le N}).
\end{align*}

$\bullet$ Continuity of $\widetilde{\RL}_{SORAS,2}$:\\
Let $\delta>0$ be a positive parameter. For $\mathcal{U} = (\mathbf{U}_0, (\mathbf{U}_i)_{1 \le i \le N}) \in \HP$, by using the fact that $(\id-\eta_i)\mathbf{U}_i = \mathbf{U}_i$ for $1 \le i \le N$ (recall that $\mathbf{U}_i \in W_{i,\gamma}$), $\mathrm{Im}(P_0-\tilde{P}_0)$ is $a$-orthogonal to $\mathrm{Im}(\id-P_0)$, the Cauchy--Schwarz inequality, Young's inequality (with parameter $\delta$), $A$-orthogonality of $I-P_0$, the bound \eqref{eq:EBoundedBytildeE} from \Cref{lemma:EandtildeE}, and Lemma~7.9 of \lecnot (page~171) we have
\begin{align*}
& a(\widetilde{\RL}_{SORAS,2}(\mathcal{U}),\widetilde{\RL}_{SORAS,2}(\mathcal{U}))\\
&= \| \RL_{SORAS,2}(\mathcal{U}) + (P_0-\tilde{P}_0) \textstyle\sum_{i=1}^{N} R_i^T D_i \mathbf{U}_i \|_{A}^{2}\\
&= \| \RL_{SORAS,2}(\mathcal{U}) \|_{A}^{2} + 2a(\RL_{SORAS,2}(\mathcal{U}),(P_0-\tilde{P}_0) \textstyle\sum_{i=1}^{N} R_i^T D_i \mathbf{U}_i))\\
& \mathrel{\phantom{=}} \mathrel+ \| (P_0-\tilde{P}_0) \textstyle\sum_{i=1}^{N} R_i^T D_i \mathbf{U}_i \|_{A}^{2}\\
&= \| Z {\mathbf U}_0 \|_{A}^{2} + \| (\id-P_0) \textstyle\sum_{i=1}^N R_i^T D_i {\mathbf U}_i \|_{A}^{2} + 2a(Z\mathbf{U}_0,(P_0-\tilde{P}_0) \textstyle\sum_{i=1}^{N} R_i^T D_i \mathbf{U}_i)\\
& \mathrel{\phantom{=}} \mathrel+ \| (P_0-\tilde{P}_0) \textstyle\sum_{i=1}^{N} R_i^T D_i \mathbf{U}_i \|_{A}^{2}\\
&\le \| Z\mathbf{U}_0 \|_{A}^{2} + \| \textstyle\sum_{i=1}^N R_i^T D_i {\mathbf U}_i \|_{A}^{2} + \delta \| Z\mathbf{U}_0 \|_{A}^{2} + \tfrac{1}{\delta} \| (P_0-\tilde{P}_0) \textstyle\sum_{i=1}^{N} R_i^T D_i \mathbf{U}_i \|_{A}^{2})\\
& \mathrel{\phantom{=}} \mathrel+ \| (P_0-\tilde{P}_0) \textstyle\sum_{i=1}^{N} R_i^T D_i \mathbf{U}_i \|_{A}^{2}\\
&\le (1+\delta) \| Z\mathbf{U}_0 \|_{A}^{2} + \left(1 + (1+\tfrac{1}{\delta}) \| P_0-\tilde{P}_0 \|_{A}^{2} \right) \| \textstyle\sum_{i=1}^{N} R_i^T D_i \mathbf{U}_i \|_{A}^{2}\\
&\le (1+\delta) \lambda_{max}(E\tilde{E}^{-1}) ( \tilde{E}\mathbf{U}_0, \mathbf{U} ) + \left(1 + (1+\tfrac{1}{\delta}) \| P_0-\tilde{P}_0 \|_{A}^{2} \right) \NC \textstyle\sum_{i=1}^{N} \| R_i^T D_i \mathbf{U}_i \|_{A}^{2}\\
&= (1+\delta) \lambda_{max}(E\tilde{E}^{-1}) ( \tilde{E}\mathbf{U}_0, \mathbf{U} )\\
& \mathrel{\phantom{=}} \mathrel+ \left(1 + (1+\tfrac{1}{\delta}) \| P_0-\tilde{P}_0 \|_{A}^{2} \right) \NC \textstyle\sum_{i=1}^{N} (A R_i^T D_i (\id-\eta_{i}) {\mathbf U}_i , R_i^T D_i (\id-\eta_{i}) {\mathbf U}_i )\\
&\le c_R \, \tilde{b}(\mathcal{U},\mathcal{U}),
\end{align*}
with
\begin{align*}
c_R = \max\left((1+\delta) \lambda_{max}(E\tilde{E}^{-1}), \left(1 + (1+\tfrac{1}{\delta}) \| P_0-\tilde{P}_0 \|_{A}^{2} \right) \NC \gamma \right).
\end{align*}
Note that this continuity constant is similar to the case of the AS algorithm, only now we have an additional factor of $\gamma$ in the second term. Again, we can minimise over $\delta>0$ enabling us to take
\begin{align}
\label{eq:c_R_inexact_soras}
c_R = \frac{\NC \gamma (1+\epsilon_{A}^2) + \lambda_{+} + \sqrt{(\NC \gamma (1+\epsilon_{A}^2)-\lambda_{+})^2 + 4 \lambda_{+} \NC \gamma \epsilon_{A}^2}}{2}.
\end{align}

$\bullet$ Stable decomposition with $\widetilde{\RL}_{SORAS,2}$:\\
For $\mathbf{U} \in \HO$ we have that
\begin{align*}
\mathbf{U} &= P_0 \mathbf{U} + (\id - P_0) \mathbf{U} = P_0 \mathbf{U} + (\id-P_0) \textstyle\sum_{i=1}^{N} R_i^T D_i R_i \mathbf{U}\\
&= P_0 \mathbf{U} + (\id-P_0) \textstyle\sum_{i=1}^{N} R_i^T D_i (\id-p_i) R_i \mathbf{U}\\
&= P_0 \mathbf{U} + (\id-P_0) \textstyle\sum_{i=1}^{N} R_i^T D_i (\id-\xi_{0i}) (\id-p_i) R_i \mathbf{U}\\
&= F \mathbf{U} + (\id-\tilde{P}_0) \textstyle\sum_{i=1}^{N} R_i^T D_i (\id-\xi_{0i}) (\id-p_i) R_i \mathbf{U},
\end{align*}
where
\begin{align}
\label{eq:FU_soras}
F \mathbf{U} = P_0 \mathbf{U} + (\tilde{P}_0-P_0) \textstyle\sum_{i=1}^{N} R_i^T D_i (\id-\xi_{0i}) (\id-p_i) R_i \mathbf{U} \in V_0.
\end{align}
Let $\mathbf{U}_0 \in \R^{\#{\mathcal N}_0}$ be such that $Z \mathbf{U}_0 = F\mathbf{U}$, then we have the decomposition
\begin{align*}
\mathbf{U} = \widetilde{\RL}_{SORAS,2}(\mathbf{U}_0,((\id-\xi_{0i}) (\id-p_i) R_i \mathbf{U})_{1 \le i \le N}) =: \widetilde{\RL}_{SORAS,2}(\mathcal{U}).
\end{align*}
We now show that this decomposition is stable, again following an analogous approach to \cite{Nataf:2020:MAR}. Firstly, note that the bound in \eqref{eq:lowerBoundFromEig-SORAS-special} applies to the domain decomposition part of $\tilde{b}(\mathcal{U},\mathcal{U})$. The remaining term in $\tilde{b}(\mathcal{U},\mathcal{U})$ corresponds to the coarse operator $\tilde{E}$ where, using the bound \eqref{eq:tildeEBoundedByE} of \Cref{lemma:EandtildeE}, we have
\begin{align*}
(\tilde{E} {\mathbf U}_0,{\mathbf U}_0) &\le \lambda_{max}(E^{-1} \tilde{E}) \| Z\mathbf{U}_0 \|_{A}^{2} = \lambda_{max}(E^{-1} \tilde{E}) \|F\mathbf{U}\|_{A}^2.
\end{align*}
Now from \eqref{eq:FU_soras}, letting $\delta>0$ be a positive parameter, making use of the Cauchy--Schwarz inequality, Young's inequality (with parameter $\delta$), Lemma~7.9 of \lecnot (page~171) gives
\begin{align*}
\|F\mathbf{U}\|_{A}^2 &\le \| P_0 \mathbf{U} \|_{A}^{2} + 2a(P_0 \mathbf{U},(\tilde{P}_0-P_0) \textstyle\sum_{i=1}^{N} R_i^T D_i (\id-\xi_{0i}) (\id-p_i) R_i \mathbf{U})\\
& \mathrel{\phantom{=}} \mathrel+ \| (\tilde{P}_0-P_0) \textstyle\sum_{i=1}^{N} R_i^T D_i (\id-\xi_{0i}) (\id-p_i) R_i \mathbf{U} \|_{A}^{2}\\
&\le (1+\delta)\| P_0 \mathbf{U} \|_{A}^{2} + (1+\tfrac{1}{\delta})\| (\tilde{P}_0-P_0) \textstyle\sum_{i=1}^{N} R_i^T D_i (\id-\xi_{0i}) (\id-p_i) R_i \mathbf{U} \|_{A}^{2}\\
&\le (1+\delta) a(\mathbf{U},\mathbf{U}) + (1+\tfrac{1}{\delta}) \epsilon_{A}^{2} \| \textstyle\sum_{i=1}^{N} R_i^T D_i (\id-\xi_{0i}) (\id-p_i) R_i \mathbf{U} \|_{A}^{2}\\
&\le (1+\delta) a(\mathbf{U},\mathbf{U}) + (1+\tfrac{1}{\delta}) \epsilon_{A}^{2} \NC \textstyle\sum_{i=1}^{N} \| R_i^T D_i (\id-\xi_{0i}) (\id-p_i) R_i \mathbf{U} \|_{A}^{2}\\
&\le (1+\delta) a(\mathbf{U},\mathbf{U}) + (1+\tfrac{1}{\delta}) \epsilon_{A}^{2} \NC \gamma \textstyle\sum_{i=1}^{N} \| (\id-\xi_{0i}) (\id-p_i) R_i \mathbf{U} \|_{B_i}^{2}\\
&\le \left( (1+\delta) + (1+\tfrac{1}{\delta}) \epsilon_{A}^{2} \NC \MC \tau \gamma \right) a(\mathbf{U},\mathbf{U}),
\end{align*}
where in the penultimate step we have made use of Lemma 6 in \cite{Nataf:2020:MAR} (applied with $A$ replaced by $D_i R_i A R_i^T D_i$ and $B$ by $\Bw$) and in the last step utilised the bound in \eqref{eq:lowerBoundFromEig-SORAS-special}.

We can minimise over the parameter $\delta$, yielding $\delta = \epsilon_{A}\sqrt{\NC\MC\tau\gamma}$, and thus
\begin{align*}
(\tilde{E} {\mathbf U}_0,{\mathbf U}_0) &\le \lambda_{max}(E^{-1} \tilde{E}) (1+\epsilon_{A}\sqrt{\NC\MC\tau\gamma})^2 a(\mathbf{U},\mathbf{U}).
\end{align*}
Combining this estimate with \eqref{eq:lowerBoundFromEig-SORAS-special} gives $\tilde{b}(\mathcal{U},\mathcal{U}) \le c_T^{-1} \, a(\mathbf{U}, \mathbf{U})$ where
\begin{align}
\nonumber
c_T &= \frac{1}{\MC\tau + \lambda_{max}(E^{-1} \tilde{E}) (1+\epsilon_{A}\sqrt{\NC\MC\tau\gamma})^2}\\
\label{eq:c_T_inexact_soras}
&= \frac{\lambda_{min}(E\tilde{E}^{-1})}{(1+\epsilon_{A}\sqrt{\NC\MC\tau\gamma})^2 + \lambda_{min}(E\tilde{E}^{-1})\MC\tau}.
\end{align}
Thus we see that the constant in \eqref{eq:c_T_inexact_soras} is given solely in terms of the constants $\NC$, $\MC$, $\tau$ and $\gamma$, and the minimal and maximal eigenvalues of $E\tilde{E}^{-1}$.

Altogether, the fictitious space lemma provides the following spectral bounds
\begin{align*}
c_T \, a(\mathbf{U},\mathbf{U}) \le a\left( \tilde{M}_{SORAS,2}^{-1} A \mathbf{U}, \mathbf{U} \right) \le c_R \, a(\mathbf{U},\mathbf{U})
\end{align*}
for all $\mathbf{U} \in \HO = \R^{\#{\mathcal N}}$, where $c_R$ and $c_T$ are given by \eqref{eq:c_R_inexact_soras} and \eqref{eq:c_T_inexact_soras} respectively, and
\begin{align*}
\tilde{M}_{SORAS,2}^{-1} &= \widetilde{\RL}_{SORAS,2} \tilde{B}^{-1} \widetilde{\RL}_{SORAS,2}^*\\
&= Z \tilde{E}^{-1} Z^T + (\id-\tilde{P}_0) \sum_{i=1}^N R_i^T \tilde{B}_i^\dag q_i R_i (\id-\tilde{P}_0^T)\\
&= \tilde{P}_0 A^{-1} + (\id-\tilde{P}_0) \sum_{i=1}^N R_i^T (\id-\eta_i) (\id-\xi_{0i}) B_i^{-1} R_i (\id-\tilde{P}_0^T).
\end{align*}
Note that in this case of inexact coarse solves we must retain the projection onto $W_{i,\gamma}$ in the form of the preconditioner.

\subsubsection*{Disclaimer}
This document provides a working draft, which details new theoretical results of interest, but is not yet fully complete in background and scope. As such, it is supplied as a pre-print draft version.

\bibliographystyle{plain}
\bibliography{preprint}

\end{document}